\documentclass{amsart} % in was article
\usepackage{amssymb}
\usepackage{amsmath}
\usepackage{amsfonts}

\begin{document}
%%%%%%%%%%%%%%%%%%%%%%%%%%%%%%%%%%%%%%%%%%%%%%%%%%%%%%%%%%%%%%%%%%%%%%
%	spaces for your own definitions follows
%%%%%%%%%%%%%%%%%%%%%%%%%%%%%%%%%%%%%%%%%%%%%%%%%%%%%%%%%%%%%%%%%%%%%%
\newtheorem{theo}{Theorem}[section]
\newtheorem{prop}[theo]{Proposition}
\newtheorem{lemma}[theo]{Lemma}
\newtheorem{exam}[theo]{Example}
\newtheorem{coro}[theo]{Corollary}
\theoremstyle{definition}
\newtheorem{defi}[theo]{Definition}
\newtheorem{rem}[theo]{Remark}

%\renewcommand{\theequation}{\mbox{\arabic{section}.\arabic{equation}}}

%letters - added these
\newcommand{\Bb}{{\bf B}}
\newcommand{\Nb}{{\bf N}}
\newcommand{\Qb}{{\bf Q}}
\newcommand{\Rb}{{\bf R}}
\newcommand{\Zb}{{\bf Z}}
\newcommand{\Ac}{{\mathcal A}}
\newcommand{\Bc}{{\mathcal B}}
\newcommand{\Cc}{{\mathcal C}}
\newcommand{\Dc}{{\mathcal D}}
\newcommand{\Fc}{{\mathcal F}}
\newcommand{\Ic}{{\mathcal I}}
\newcommand{\Jc}{{\mathcal J}}
\newcommand{\Lc}{{\mathcal L}}
\newcommand{\Oc}{{\mathcal O}}
\newcommand{\Pc}{{\mathcal P}}
\newcommand{\Sc}{{\mathcal S}}
\newcommand{\Tc}{{\mathcal T}}
\newcommand{\Uc}{{\mathcal U}}
\newcommand{\Vc}{{\mathcal V}}

\newcommand{\ax}{{\rm ax}}
\newcommand{\Acc}{{\rm Acc}}
\newcommand{\Act}{{\rm Act}}
\newcommand{\ded}{{\rm ded}}
\newcommand{\Gm}{{$\Gamma_0$}}
\newcommand{\ID}{{${\rm ID}_1^i(\Oc)$}}
\newcommand{\PA}{{\rm PA}}
\newcommand{\ACA}{{${\rm ACA}^i$}}
\newcommand{\RefP}{{${\rm Ref}^*({\rm PA}(P))$}}
\newcommand{\RefS}{{${\rm Ref}^*({\rm S}(P))$}}
\newcommand{\Rfn}{{\rm Rfn}}
\newcommand{\tar}{{\rm Tarski}}
\newcommand{\UNFA}{{${\mathcal U}({\rm NFA})$}}

\author{Nik Weaver}

\title [Constructive concepts]
       {Reasoning about constructive concepts}

\address {Department of Mathematics\\
          Washington University in Saint Louis\\
          Saint Louis, MO 63130}

\email {nweaver@math.wustl.edu}

\date{\em December 12, 2013}

\maketitle

%%%%%%%%%%%%%%%%%%%%%%%%%%%%%%%%%%%%%%%%%%%%%%%%%%%%%%%%%%%%%%%%%%%%%%%
%	Please insert the article body now
%%%%%%%%%%%%%%%%%%%%%%%%%%%%%%%%%%%%%%%%%%%%%%%%%%%%%%%%%%%%%%%%%%%%%%%

\section{}

Second order quantification becomes problematic when a quantified concept
variable is supposed to function predicatively. There is a use/mention issue.

The distinction that comes into play is illustrated in Tarski's classic
biconditional
\cite{Tar}
$$\mbox{``Snow is white'' is true $\leftrightarrow$ snow is white.}$$
The expression ``snow is white'' is mentioned on the left side; there
it is linguistically inert and appears only as an object under discussion.
On the right side it is in use and has assertoric force.

To see the sort of problem that can arise, suppose we try to define the
truth of an arbitrary sentence by saying
$$\mbox{$\ulcorner A\urcorner$ is true $\leftrightarrow$ $A$,}\eqno{(*)}$$
where $A$ is taken to range over all sentences and the use/mention distinction
is indicated using corner brackets. Why is this one statement not a
global definition of truth?

The answer depends on whether the variable $A$ in ($*$) is understood
as schematic or as being implicitly quantified. If we interpret ($*$)
schematically, that is, as a sort of template which is not itself an
assertion but which becomes one when any sentence is substituted for $A$,
then it cannot be a definition of truth since it is the wrong kind of
object (a definition can be asserted, a template cannot). We could still
use it as a tool to construct truth definitions in limited settings:
given any target language, the conjunction of all substitution instances
of ($*$), as $A$ ranges over all the sentences of the language, would
define truth for that language. But this conjunction generally will not
belong to the target language, so we cannot construct a global definition
of truth for all sentences in this way. More to the point, we cannot
use this approach to build a language in which we have the ability to
discuss the truth of any sentence in that language.

The expression ($*$) could be directly interpreted as a truth definition for
all sentences only by universally quantifying the variable $A$. However,
this is impossible for straightforward syntactic reasons. In the quantifying
phrase ``for every sentence $A$'' the symbol $A$ has to represent a
mention, not a use, of an arbitrary sentence, since here the arbitrary
sentence is being referred to and not asserted. But we need $A$ to represent
a use in the right side of the biconditional. So there is no meaningful way
to quantify over $A$ in ($*$). This expression can only be understood as
schematic.

(The general principle is that a schematic expression can be obtained by
omitting any part of a well-formed sentence, but we can only quantify
over omitted noun phrases. ``Snow is white'' can
be schematized to either ``$x$ is white'' or ``Snow $C$'', but
``$(\exists x)(x$ is white$)$'' is grammatical while
``$(\exists C)($snow $C)$'' is not. In ($*$) the omission is
not nominal.)

Since a quantified variable can only represent a mention of an arbitrary
sentence, what we would need in order to formulate a global definition of
truth is a disquotation operator $\urcorner\cdot\ulcorner$. Then we could
let the variable $A$ refer to an arbitrary sentence and write
$$\mbox{For every sentence $A$, $A$ is true $\leftrightarrow$
$\urcorner A\ulcorner$.}$$
In other words, we need some way to convert mention into use. But that
is exactly what having a truth predicate does for us. The way we convert
a mention of the sentence ``snow is white'' into an actual assertion that
snow is white is by saying that the mentioned sentence is true. So in order
to state a global definition of truth we would need to, in effect, already
have a global notion of truth.

This is why no variant of ($*$) can succeed in globally defining truth.
A quantified variable representing an arbitrary sentence in general cannot
be invested with assertoric force unless we possess a notion of truth that
applies to all sentences, but writing down a global definition of truth
requires us to already be able to construe such a variable as
having assertoric force.

Truth seems unproblematic because in any particular instance it really is
unproblematic. Any meaningful sentence can be substituted for $A$ in ($*$)
with straightforward results. But the essentially grammatical problem with
quantifying over $A$ in ($*$) is definitive. There is no way to use this
schematic condition to globally define truth, and we can be quite certain
of this because any predicate which globally verified ($*$) would engender
a contradiction. It would give rise to a liar paradox.

\section{}

Similar comments can be made about what it means for an
object to fall under a concept. Just as with truth, there is no
difficulty in defining this relation in any particular case. For
instance, we can define what it means to fall under the concept
{\it white} by saying
$$\mbox{The object $x$ falls under the concept {\it white} $\leftrightarrow$
$x$ is white.}$$
But if we try to characterize the falling under relation globally by saying
$$\mbox{The object $x$ falls under the concept
$\ulcorner C\urcorner$ $\leftrightarrow$
$Cx$}\eqno{(\dag)}$$
then, just as with ($*$), a use/mention conflict arises when we try to
quantify over $C$. We can use ($\dag$) as a template to produce a falling
under definition for any particular concept; we can even take the
conjunction of the substitution instances of ($\dag$) as $C$ ranges
over all the concepts expressible in a given language, and thereby obtain
a falling under definition for that language, but this definition could
not itself belong to the language in question. As with ($*$),
in order to put ($\dag$) in a form that would allow $C$ to be
quantified, so that it could have global force,
we would need some device for converting a mention of an arbitrary
concept into a use of that concept. But that is exactly what the falling under
relation does for us. That is to say, we need to already have a global
notion of falling under before we can use ($\dag$) to define falling
under globally. Thus, no variant of ($\dag$) can succeed in globally
defining a falling under relation.

In fact, the twin difficulties with truth and falling under are not just
analogous, they are effectively equivalent. If we had a globally applicable
truth predicate, we could use it to define a global notion of falling under,
viz., ``An object falls under the concept $C$ $\leftrightarrow$ the
atomic proposition formed from a name for that object and $C$ is
true.'' Conversely, given a globally applicable falling
under relation, truth could be defined globally by saying ``The sentence
$A$ is true $\leftrightarrow$ every object falls under the predicate formed
by concatenating `is such that' with $A$''. We can now see that truth and
falling under are practically identical notions. Falling under is to
formulas with one free variable what truth is to sentences.

However, there is one striking difference between the two cases. The
globally problematic nature of truth does not have any
immediate implications for our understanding of logic, but, in contrast,
the globally problematic nature of falling under has severe consequences
for general second order quantification. As we have seen, expressions like
$(\exists C)Cx$ are, taken at face value, syntactically ill-formed.
The quantified concept variable $C$ cannot function predicatively because its
appearance in the quantifying phrase is a mention, not a use. In order to
make sense of expressions like this we need a global device for converting
a mention of an arbitrary concept into a use of that concept, which is just
to say that we need a global notion of falling under. But it should now be
clear that a global notion of falling under, in the form of a relation which
satisfies ($\dag$) for every concept $C$, is something we do not and cannot
have. (Cannot, because it would give rise to Russell's paradox.) Now, if
$C$ were restricted to range over
only those concepts appearing in some given language, then we could use
($\dag$) as a template to define falling under for those concepts and
thereby render quantification over them meaningful. But sentences
employing these quantifiers would not belong to the target
language, so this approach cannot be used to make sense of unrestricted
second order quantification. Specifically, it cannot be used to build a
language in which we have the ability to quantify over all concepts
expressible in that language while allowing the quantified concept
variable to predicate.

Thus, we are not straightforwardly able to assign meaning to statements in
which a quantified concept variable is supposed to function predicatively.

\section{}

This negative conclusion is unsatisfying because our syntactic considerations
forbid not only the paradoxical global definitions of truth and
falling under which we want to exclude, but also other global statements
which appear to be meaningful. For instance, we have remarked that in
limited settings truth and falling under are unproblematic, but it is not
obvious how to formalize this claim itself. We cannot say that for any
sentence $A$ there is a predicate $T$ such that $T(A) \leftrightarrow A$,
for the same reason that we cannot quantify over $A$ in ($*$): the mentions
of $T$ and $A$ in the quantifying phrases are followed by uses in the
expression $T(A) \leftrightarrow A$.  Reformulations like ``such that
$T(A) \leftrightarrow A$ holds'' or ``such that
$T(A) \leftrightarrow A$ is the case'' accomplish nothing because they
merely employ synonyms for truth. But this prohibition is confusing
because there clearly is some sense in which it is correct, even trivially
correct, to say that a truth definition can be given for any meaningful
sentence. We either have to adopt the mystical (and rather self-contradictory)
view that this is a fact which cannot be expressed, or else find some
legitimate way to affirm it.

The way forward is to recognize that truth and falling under do make sense
globally, but as constructive, not classical, notions. In both cases we
can recognize an indefinitely extensible quality: we are able to produce
partial classical characterizations of truth and falling under, but any
such characterization can
be extended. This fits with the intuitionistic conception of
mathematical reality as something which does not have a fixed global
existence but instead is open-ended and can only be constructed in
stages. The intuitionistic account may or may not be valid
as a description of mathematics, but it unequivocally does capture the
fundamental nature of truth and falling under. On pain of contradiction,
these notions do not enjoy a global classical existence. However,
they can indeed be built up in an open-ended sequence of stages.

The central concept in constructive mathematics is proof, not truth.
And this is just the linguistic resource we need to make
sense of second order quantification. Writing $\Box A$ for
``$A$ is provable'' and $p \, \vdash A$ for ``$p$ proves $A$'', we have
$$\Box A \leftrightarrow (\exists p)(p\, \vdash A).$$
Note that this expression can be universally quantified because no
appearance of $A$ is assertoric, so that it is a legitimate definition
of the box operator. Note also that there is no question about
formulating a global definition of the proof relation, as this is a
primitive notion which we do not expect to define in any simpler terms.
We can therefore, following the intuitionists, give a global constructive
definition of truth by saying
$$\mbox{$A$ is constructively true $\leftrightarrow$
$A$ is provable}\eqno{(**)}$$
and we can analogously give a global constructive definition of falling
under by saying
$$\mbox{$x$ constructively falls under $C$ $\leftrightarrow$
$C(x)$ is provable.}\eqno{(\dag\dag)}$$

More generally, we can use provability to repair use/mention problems in
expressions that quantify over all concepts. Such expressions can be
interpreted constructively, and the self-referential capacity of global
second order quantification makes it unreasonable to demand a classical
interpretation. In particular, we can solve
the problem raised at the beginning of this section.
The way we say that any sentence can be given a truth definition is:
for every sentence $A$ there is a predicate $T$ such that the assertion
$T(A) \leftrightarrow A$
is provable. Again, no appearance of $A$ or $T$ is assertoric, so the
quantification is legitimate. More substantially, we can affirm that
for any language $\Lc$ there is a predicate $T_\Lc$ such that inserting
any sentence of $\Lc$ in the template ``$T_\Lc(\cdot) \leftrightarrow \cdot$''
yields a provable assertion. Thus, there is a global
constructive definition of truth, and there are local classical definitions
of truth, but the global affirmation that these local classical
definitions always exist is constructive.

We can make the same points about falling under; here too we have both
local classical and global constructive options. The new feature in
this setting is that no local classical definition of falling under can
be used to make sense of sentences in which quantified concept variables
are supposed to function predicatively. In order to handle this problem
we require a globally applicable notion of falling under, which means that
the classical option is unworkable. We have to adopt a constructive
approach.

\section{}

The global constructive versions of truth and falling under are not
obviously paradoxical because the biconditional $\Box A \leftrightarrow A$
is not tautological. We cannot simply assume that asserting $A$ is
equivalent to asserting that $A$ is provable. The extent to which this
law holds is a function of both the nature of provability and the
constructive interpretation of implication. This issue is analyzed in
\cite{W1} (see also \cite{W2}); we find that the law
\begin{quote}
(1) $A \to \Box A$
\end{quote}
is valid but the converse inference of $A$ from $\Box A$ is legitimate only
as a deduction rule, not as an implication. Although the law $\Box A \to A$
is superficially plausible, its justification is in fact subtly circular.

The relation of $\Box$ to the standard logical constants, interpreted
constructively, is also investigated in \cite{W1}; we find that the laws
\begin{quote}
(2) $\Box(A \wedge B) \leftrightarrow (\Box A \wedge \Box B)$

\noindent (3) $\Box(A \vee B) \leftrightarrow (\Box A \vee \Box B)$

\noindent (4) $\Box((\exists x)A) \leftarrow (\exists x)\Box A$

\noindent (5) $\Box((\forall x)A) \to (\forall x)\Box A$

\noindent (6) $\Box(A \to B) \to (\Box A \to \Box B)$
\end{quote}
are all generally valid. There is no special law for negation; we
take $\neg A$ to be an abbreviation of $A \to \bot$ where $\bot$
represents falsehood, so using (6) we can say, for instance,
$$\Box(\neg A) \leftrightarrow \Box(A \to \bot)
\to (\Box A \to \Box \bot).$$
But $\Box(\neg A)$ is not provably equivalent to $\neg\Box A$ in general.

We can now present a formal system for reasoning about concepts that
allows quantified concepts to predicate. The language is the language of
set theory, augmented by the logical constant $\Box$. Formulas are built
up in the usual way, with the one additional clause that if $A$ is a
formula then so is $\Box A$.

The variables are taken to range over concepts and $\in$
is read as ``constructively falls under''. Thus no appearance of a
variable in any formula is assertoric and we can sensibly quantify
over any variable in any formula. The system employs the usual axioms
and deduction rules of an intuitionistic predicate calculus with equality,
together with the axioms (1) -- (6) above, the deduction rule which infers
$A$ from $\Box A$, the extensionality axiom
\begin{quote}
(7) $x = y \leftrightarrow (\forall u)(u \in x \leftrightarrow u \in y)$,
\end{quote}
and the comprehension scheme
\begin{quote}
(8) $(\exists x)(\forall r)(r \in x \leftrightarrow \Box A)$
\end{quote}
where $x$ can be any variable and $A$ can be any formula in which $x$
does not appear freely. (In this scheme the variable $r$ is fixed.) The
motivation for the comprehension scheme is that any formula defines a
concept (possibly with parameters, if $A$ contains free variables
besides $r$), and what it means to constructively fall under that
concept is characterized by ($\dag\dag$). This is why we
need the ability to explicitly reference the notion of provability.
The ex falso law can be justified in this setting by taking $\bot$ to
stand for the assertion $(\forall x,y)(x \in y)$.

We call the formal system described in this section CC (Constructive
Concepts). This is a ``pure'' concept system in the sense that there
are no objects besides concepts. Alternatively, we could (say) take
the natural numbers as given and write down a version of second order
arithmetic in which the set variables are interpreted as concepts.
From a predicative point of view a third order system, with number
variables, set variables, and concept variables, would also be natural
\cite{W0}.

\section{}

The system CC accomodates global reasoning about concepts. For instance,
using comprehension we can define the concept {\it concept which does not
provably fall under itself}. Denoting this concept $R$, we have
$$r \in R \qquad \leftrightarrow \qquad
\Box (r \not\in r).$$
Assuming $R \not\in R$ then yields $\Box(R \not\in R)$ by axiom (1), which
entails $R \in R$ by the definition of $R$. This shows that $R \not\in R$
is contradictory, so we conclude
$\neg(R \not\in R)$. On the other hand, assuming $R \in R$ immediately
yields $\Box (R \not\in R)$; but since $R \in R$ also implies
$\Box (R \in R)$, we infer $\Box \bot$. So we have $R \in R \to \Box \bot$.
In the language of \cite{W1}, the assertion $R \not\in R$ is false
and the assertion $R \in R$ is weakly false.

Thus, we can reason in CC about apparently paradoxical concepts and reach
substantive conclusions. But no contradiction can be derived, as we will
now show. (The proof of the following theorem is similar to the proof of
Theorem 6.1 in \cite{W1}.)

\begin{theo}
CC is consistent.
\end{theo}

\begin{proof}
We begin by adding countably many constants to the language of $CC$.
Let ${\mathcal L}$ be the smallest language which contains the language
of CC and which contains, for every formula $A$ of ${\mathcal L}$ in
which no variable other than $r$ appears freely, a constant symbol $c_A$.
Observe that ${\mathcal L}$ is countable.

We define the level $l(A)$ of a formula $A$ of ${\mathcal L}$ as follows.
The level of every atomic formula and every formula of the form $\Box A$ is 1.
The level of $A \wedge B$, $A \vee B$, and $A \to B$ is $\max(l(A), l(B)) + 1$.
The level of $(\forall x)A$ and $(\exists x)A$ is $l(A) + 1$.

Now we define a transfinite sequence of sets of sentences $F_\alpha$.
These can be thought of as the sentences which we have determined not
to accept as true. The definition of $F_\alpha$ proceeds by induction
on level. For each $\alpha$ the formula $\bot$ belongs to $F_\alpha$;
$c_B \in c_A$ belongs to $F_\alpha$ if $A(c_B)$ belongs to $F_\beta$ for
some $\beta < \alpha$; $c_A = c_{A'}$ belongs to $F_\alpha$ if for some
$c_B$, one but not both of $A(c_B)$ and $A'(c_B)$ belongs to $F_\beta$
for some $\beta < \alpha$; and $\Box A$ belongs to $F_\alpha$ if $A$
belongs to $F_\beta$ for some $\beta < \alpha$. (Recall that the constants
$c_A$ are only defined for formulas $A$ in which no variable other than $r$
appears freely. So expressions like $A(c_B)$ are unambiguous.) For levels
higher than 1, we apply the following rules.
$A \wedge B$ belongs to $F_\alpha$ if either $A$ or $B$ belongs to
$F_\alpha$. $A \vee B$ belongs to $F_\alpha$ if both $A$ and $B$ belong
to $F_\alpha$. $(\forall x)A$ belongs to $F_\alpha$ if $A(c_B)$ belongs
to $F_\alpha$ for some constant $c_B$, and $(\exists x)A$ belongs to
$F_\alpha$ if $A(c_B)$ belongs to $F_\alpha$ for every constant $c_B$.
(Observe here that if $(\forall x)A$ is a sentence then $A$ can contain
no free variables other than $x$, so again the expression $A(c_B)$ is
unambiguous.)
Finally, $A \to B$ belongs to $F_\alpha$ if there exists $\beta \leq \alpha$
such that $B$ belongs to $F_\beta$ but $A$ does not belong to $F_\beta$.

Since the language ${\mathcal L}$ is countable and the sequence $(F_\alpha)$
is increasing, this sequence must
stabilize at some countable stage $\alpha_0$. It is obvious that
$\bot$ belongs to $F_{\alpha_0}$. The proof is completed by checking that
the universal closure of no axiom of CC belongs to $F_{\alpha_0}$,
and that the set of formulas whose universal closure does not belong to
$F_{\alpha_0}$ is stable under the deduction rules of CC. This is
tedious but straightforward.
\end{proof}

\section{}

The system CC gives correct expression to Frege's idea of formalizing
reasoning about arbitrary concepts. Frege was impeded by the fact that
the global notion of falling under is inherently constructive;
treating this notion as if it were classical is the fatal mistake
which gives rise to Russell's paradox. We can locate the essential
error in Frege's analysis not in his Basic Law V, or any of his other
axioms, but rather in his use of a language whose cogency depends
on a fictitious global classical notion of falling under.

Analyzing the proof theoretic strength of CC will show us the degree
to which it is possible, as Frege hoped, to base mathematical reasoning
on the pure logic of concepts. The result is disappointing.
The simplicity of the consistency proof given in Theorem 5.1
already reveals that CC must be a very weak system. We now present two
positive results which show how (conservative extensions of) CC can in
a certain sense interpret more standard formal systems in which the box
operator does not appear.

The relevant sense is the notion of {\it weak interpretation} introduced in
\cite{W1}. We say that a theory ${\mathcal T}_2$ in which we are able to
reason about provability weakly interprets another theory ${\mathcal T}_1$
in the same language minus the box operator if every theorem of
${\mathcal T}_1$ is a theorem of ${\mathcal T}_2$ with all boxes
deleted. Observe that deleting all boxes in all theorems of CC yields
an inconsistency: as we saw earlier, we can prove in CC the existence
of a concept $R$ which satisfies both $\neg\neg (R \in R)$ and
$R \in R \to \Box \bot$, and deleting the box in the second formula
produces the contradictory conclusions $\neg\neg (R \in R)$ and
$\neg(R \in R)$.
Notwithstanding this phenomenon, no inconsistent theory can be weakly
interpreted in CC. This is because $\bot$ is a theorem of every
inconsistent theory, and weak interpretability would imply that $\Box^k \bot$
must be a theorem of CC for some value of $k$. Since CC implements the
deduction rule which infers $A$ from $\Box A$, this would then imply
that $\bot$ is a theorem of CC, i.e., that CC is inconsistent.

The first system we consider, ${\rm Comp}({\rm PF}_{\mathcal T}) + {\rm D}$,
was discussed in \cite{FH}, where its consistency was proven. Here we show
that the intuitionistic version of this system is weakly interpretable in
an extension of CC by definitions.

${\rm Comp}({\rm PF}_{\mathcal T}) + {\rm D}$ is a positive set
theory. Its language is the ordinary language of set theory augmented
by terms which are generated in the following way. Any variable is a term;
if $s$ and $t$ are terms then $s \in t$ and $s = t$ are positive
formulas; if $A$ and $B$ are positive formulas then $A \wedge B$,
$A \vee B$, $(\forall x)A$, and $(\exists x)A$ are positive
formulas; if $A$ is a positive formula and $x$ is a variable then
$\{x:A(x)\}$ is a term whose variables are the free variables of $A$
other than $x$. The system consists of the comprehension scheme
$$y \in \{x:A(x)\} \leftrightarrow A(y),$$
where $A$ is a positive formula and $x$ and $y$ are variables, together
with the axiom D which states
$$(\exists x,y)(x \neq y).$$

The desired conservative extension CC${}'$ of CC is obtained by recursively
adding, for every formula $A$ and variables $x$ and $y$, the term
$\{x: \Box A(x)\}$ (whose variables are
the free variables of $A$ other than $x$) together with the axiom
$$y \in \{x: \Box A(x)\} \leftrightarrow \Box A(y).$$

Say that a formula is {\it increasing}
if no implication appears in the premise of any other
implication. Note that since we take $\neg A$ to be an
abbreviation of $A \to \bot$, this also means that an increasing
formula cannot position a negation within the premise of any implication,
nor can it contain the negation of any implication.

Observe that the axiom $y \in \{x: \Box A(x)\} \leftrightarrow \Box A(y)$ is
increasing if $A$ is positive, and the formula $(\exists x,y)(x = y \to
\Box\bot)$, which is easily provable in CC, is also increasing. Since
removing all boxes from these formulas recovers the axioms of
${\rm Comp}({\rm PF}_{\mathcal T}) + {\rm D}$, the following result
is now a consequence of (\cite{W1}, Corollary 7.3).

\begin{theo}
CC$\,{}'$ weakly interprets intuitionistic
${\rm Comp}({\rm PF}_{\mathcal T}) + {\rm D}$.
\end{theo}

It is interesting to note that the extensionality axiom of CC is not
increasing, so that we cannot weakly interpret intuitionistic
${\rm Comp}({\rm PF}_{\mathcal T}) + {\rm EXT} + {\rm D}$ in CC${}'$.
The latter theory is in fact inconsistent \cite{FH}.

We can also show that a different extension of CC by definitions weakly
interprets intuitionistic second order Peano arithmetic minus the induction
axiom. The extension is defined by adding a constant symbol 0 which satisfies
$$r \in 0 \leftrightarrow \Box \bot,$$
a unary function symbol $S$ which satisfies
$$r \in Sx \leftrightarrow \Box(r = x),$$
and a constant symbol $\omega$ which satisfies
$$r \in \omega \leftrightarrow
\Box(\forall z)[(0 \in z \wedge (\forall x)(x \in z \to Sx \in z))
\to r \in z].$$
The following formulas are easily proven in the resulting extension CC${}''$:
\begin{quote}
$0 \in \omega$;

\noindent $x \in \omega \to Sx \in \omega$;

\noindent $Sx = 0 \to \Box \bot$;

\noindent $Sx = Sy \to \Box(x = y)$;

\noindent $(0 \in z \wedge (\forall x)(x \in z \to Sx \in z)) \to
(\forall y)(y \in \omega \to \Box(y \in z))$.
\end{quote}
Since the first four of these formulas are increasing, the claimed result
again follows from (\cite{W1}, Corollary 7.3).

\begin{theo}
CC$\,{}''$ weakly interprets intuitionistic second order Peano arithmetic
minus induction.
\end{theo}

Since the induction axiom is not increasing it has to be excluded from
this result. Thus, although CC proves a version of full second order
induction, it nonetheless appears to possess only meager number theoretic
resources.

%##

\end{document}